\documentclass[12pt]{article}
\usepackage{amssymb,amsmath, amsthm, dsfont}
\usepackage[mathscr]{euscript}
\usepackage{fullpage}
\usepackage{indentfirst}
\usepackage{hyperref}
\usepackage{mathtools}

\usepackage{etoolbox}
\usepackage{hyperref}

\makeatletter
\newcounter{author}
\renewcommand*\author[1]{%
  \stepcounter{author}%
  \ifnum\c@author=1
  \gdef\@author{#1}%
  \else
  \xdef\@author{\unexpanded\expandafter{\@author\and#1}}%
  \fi
  \csgdef{author@\the\c@author}{#1}}
\newcommand*\email[1]{%
  \csgdef{email@\the\c@author}{#1}}
\newcommand*\address[1]{%
  \csgdef{address@\the\c@author}{#1}}
\AtEndDocument{%
  \xdef\author@count{\the\c@author}%
  \c@author=1
  \print@authors}
\newcommand*\print@authors{%
  \ifnum\c@author>\author@count
  \else
  \print@author{\the\c@author}%
  \advance\c@author by 1
  \expandafter\print@authors
  \fi}
\newcommand*\print@author[1]{%
  \par\medskip
  \begin{tabular}{@{}l@{}}%
    \textsc{\csuse{author@#1}}\\
    \csuse{address@#1}\\
    \textit{E-mail address}:
    \href{mailto:\csuse{email@#1}}{\csuse{email@#1}}
  \end{tabular}}
\makeatother

\newtheorem{theorem}{Theorem}

\newtheorem*{remark*}{Remark}

\newcommand{\mexp}[1]{\mathbb{E}[#1]}

\title{Modulus of continuity for a martingale sequence}
\date{}

\author{Azat Miftakhov}
\address{Moscow State University, Faculty of Mechanics and Mathematics,\\ 1 Leninskiye Gory 119991, Moscow, Russia}
\email{miftakhov-af@tutanota.com}

\begin{document}
\maketitle
\begin{abstract}
  Given a martingale sequence of random fields that satisfies a natural assumption of boundedness, it is shown that the pointwise limit of this sequence can be modified in such a way that a certain class of moduli of continuity is preserved. That is, if every element of the sequence admits a given modulus of continuity, one can construct a modification of the limiting random field so that this new field also admits the same modulus of continuity. Additionally, it is shown that requiring further smoothness and a stronger notion of boundedness for the original sequence guarantees further smoothness of the limiting field and a stronger mode of convergence to this limit. Moreover, the modulus of continuity is also  preserved for the derivatives.
\end{abstract}

\noindent \textbf{1.} Let~$\theta: \mathbb{R}_+ \to \mathbb{R}_+$ be a \textit{modulus of continuity} which is continuous, increasing, and subadditive. In other words, $\theta$ is a continuous increasing function,~$\theta(0)=0$, that satisfies~$\theta(x+y) \le \theta(x)+\theta(y)$ for all~$x,y \in \mathbb{R}_+$. Further we are interested exclusively in continuous, increasing, and subadditive moduli, and for the sake of brevity they are referred to simply as moduli. A canonical example of a modulus of continuity is given by~$\theta(x) = x^\alpha$, $\alpha\in(0,1]$, which describe the property of H\"{o}lder continuity.

For any function~$f$ on a compact domain~$E \subset \mathbb{R}^d$, we say that it \textit{admits} the modulus of continuity~$\theta$ if and only if
\begin{equation}
  \label{intro_eq1}
  \sup\limits_{x \ne y} \frac{|f(x)-f(y)|}{\theta(|x-y|)} < +\infty.
\end{equation}

We are going to study~\eqref{intro_eq1} for random fields that are elements of a martingale sequence. To simplify the technical matters we bound ourselves to considering only discreet-time martingales. The proof can be easily modified to cover the continuous case as well, however one needs to impose additional technical conditions.

Consider a filtered probability space~$(\Omega, \mathcal{F}, \{\mathcal{F}_n\}_{n \ge 0}, \mathcal{P})$. For the sake of convenience we assume that both the probability space and the filtration are complete. In this setting the following statement holds.
\begin{theorem}
  \label{th1}
  Let~$\{(\xi_n(x), x \in E)\}_{n \ge 0}$ be a sequence of random fields such that their realizations admit a modulus of continuity~$\theta$ almost surely. Set
  \begin{equation}
    \label{mn_def}
    M_n \overset{\mathrm{def}}{=} \sup\limits_{x}|\xi_n(x)| + \sup\limits_{x \ne y} \frac{|\xi_n(x)-\xi_n(y)|}{\theta(|x-y|)},
  \end{equation}
  and assume that~$(\xi_n(x),\{\mathcal{F}_n\}_{n\ge 0})$ is a martingale for every fixed~$x \in E$. If
  \begin{equation}
    \label{th_cond}
    \sup\limits_n \mexp{M_n} < +\infty,
  \end{equation}
  then there exist a random field~$(\xi(x), x \in E)$ such that its realizations admit the modulus of continuity~$\theta$ almost surely and such that
  \begin{equation}
    \label{th_conv}
    (\xi_n(x), x \in E) \underset{n \to \infty}{\longrightarrow} (\xi(x), x \in E)
  \end{equation}
   pointwise almost surely.
 \end{theorem}
 \begin{proof}
   Clearly, $(M_n, \{\mathcal{F}_n\}_{n\ge 0})$ is a submartingale. The condition~\eqref{th_cond} means that this submartingale is bounded. By the classical Doob's martingale convergence theorem (e.g., see~\cite{Chung}) one has that
   \begin{equation}
     M_n \underset{n \to \infty}{\longrightarrow} M
   \end{equation}
   almost surely, where~$M$ is a random variable with expectation~$\mexp{M} < +\infty$. Condition~\eqref{th_cond} also implies that the martingale~$(\xi_n(x),\{\mathcal{F}_n\}_{n\ge 0})$ is bounded for every fixed~$x \in E$. Again, Doob's martingale convergence theorem yields
   \begin{equation}
     \xi_n(x) \underset{n \to \infty}{\longrightarrow} \tilde{\xi}(x)
   \end{equation}
   almost surely, for some random variable~$\tilde{\xi}(x)$ with expectation~$\mexp{\tilde{\xi}(x)}<+\infty$. In this way one can define the random field~$(\tilde{\xi}(x), x \in E)$. Note, however, that neither can we claim that realizations of this field admit~$\theta$ almost surely, nor can we claim~$(\xi_n(x), x \in E) \underset{n \to \infty}{\longrightarrow} (\tilde{\xi}(x), x \in E)$ pointwise almost surely. The almost sure convergence merely takes place for every fixed~$x \in E$, and the exceptional set of full-measure, in fact, depends upon~$x \in E$. We are going to construct a modification of~$(\xi(x), x\in E)$, such that it admits~$\theta$, and prove the corresponding convergence.

   Let~$A$ be a dense countable subset of~$E$. Since~$A$ is countable,
   \begin{equation}
     (\xi_n(x), x \in A) \underset{n \to \infty}{\longrightarrow} (\tilde{\xi}(x), x\in A)
   \end{equation}
   pointwise almost surely. Consequently, using~\eqref{mn_def} and passing to the limit as~$n\to \infty$ in the inequality
   \begin{equation}
     |\xi_n(x) - \xi_n(y)| \le M_n \theta(|x-y|), \quad x,y \in A,
   \end{equation}
   we obtain
   \begin{equation}
     \label{eq1}
     |\tilde{\xi}(x) - \tilde{\xi}(y)| \le M \theta(|x-y|)
   \end{equation}
   for all~$x,y \in A$ almost surely. In other words, realizations of~$(\tilde{\xi}(x) , x \in A)$ admit the modulus~$\theta$ almost surely.
   
   Define~$(\xi(x), x \in E)$ by
   \begin{equation}
     \xi(x) \overset{\mathrm{def}}{=} \inf\limits_{y \in A} \left(\tilde{\xi}(y) + M\theta(|x-y|)\right).
   \end{equation}

   We need to show that realizations of~$(\xi(x), x \in E)$ also admit~$\theta$ almost surely. First, suppose~$x \in E$ and~$y \in A$. The chain of inequalities
   \begin{equation}
     \begin{aligned}
       - M \theta(|x-y|) &\le M \inf\limits_{u \in A} \left(\theta(|x-u|) - \theta(|u-y|)\right) \\
       &\le \inf\limits_{u \in A} \left(\tilde{\xi}(u) - \tilde{\xi}(y) + M \theta(|x-u|)\right) \le M \theta(|x-y|),
     \end{aligned}
   \end{equation}
   where the first one follows by subadditivity and monotonicity of~$\theta$, gives us
   \begin{equation}
     \label{eq2}
     |\xi(x) - \tilde{\xi}(y)| = \left|\inf\limits_{u \in A} \left(\tilde{\xi}(u) - \tilde{\xi}(y) + M\theta(|x-u|)\right)\right| \le M \theta(|x-y|),
   \end{equation}
for all~$x \in E$ and~$y \in A$ almost surely. In particular, we see that~$\xi(x) = \tilde{\xi}(x)$ for all~$x \in A$ almost surely.

   Next, for any~$x,y \in E$ there exist sequences~$\{x_k\} \subset A$ and $\{y_k\} \subset A$ such that~$x_k \to x$ and~$y_k \to y$ as~$k \to \infty$. The triangle inequality and the formulas~\eqref{eq1} and~\eqref{eq2} yield
   \begin{equation}
     \begin{aligned}
       |\xi(x) - \xi(y)| &\le |\xi(x) - \tilde{\xi}(x_k)| + |\tilde{\xi}(x_k) - \tilde{\xi}(y_k)|+ |\tilde{\xi}(y_k) - \xi(y)| \\
       &\le M\big(\theta(|x-x_k|) + \theta(|x_k-y_k|) + \theta(|y_k -y|)\big).
     \end{aligned}
   \end{equation}
   And passing to the limit as~$k \to \infty$ and using the continuity of~$\theta$, we arrive at
   \begin{equation}
     \label{eq3}
     |\xi(x) - \xi(y)| \le M \theta(|x-y|)
   \end{equation}
   for all~$x,y \in E$ almost surely. This shows that realizations of~$(\xi(x), x \in E)$ admit~$\theta$ almost surely.

   The final step is to establish the pointwise convergence in~\eqref{th_conv}. Let~$\{x_k\}\subset A$ be a sequence such that~$x_k \to x$ as~$k \to \infty$. The formulas~\eqref{mn_def} and~\eqref{eq3}, along with the triangle inequality, lead us to
   \begin{equation}
     \begin{aligned}
       |\xi_n(x) - \xi(x)| &\le |\xi_n(x) - \xi_n(x_k)| + |\xi_n(x_k) - \xi(x_k)| + |\xi(x_k) - \xi(x)|\\
       &\le M_n \theta(|x-x_k|) + |\xi_n(x_k) - \xi(x_k)| + M \theta(|x_k-x|),
     \end{aligned}
   \end{equation}
   which holds for all~$x \in E$ almost surely. Passing to the limit, first as~$n \to \infty$ and then as~$k \to \infty$, and using the continuity of~$\theta$ yield~\eqref{th_conv} pointwise almost surely as desired. Due to the uniqueness of the limit it is also clear that~$(\xi(x), x \in E)$ is a modification of the field~$(\tilde{\xi}(x), x \in E)$. This completes the proof.
 \end{proof}

 \begin{remark*}
   Note that since~$E$ is bounded, the theorem also holds if one uses
   \begin{equation}
     \tilde{M}_n \overset{\mathrm{def}}{=} |\xi_n(0)| + \sup\limits_{x \ne y} \frac{|\xi_n(x)-\xi_n(y)|}{\theta(|x-y|)}.
   \end{equation}
   instead of~$M_n$. Indeed, there exists a (non-random) constant~$C>0$ such that
   \begin{equation}
     C M_n \le \tilde{M}_n\le M_n,
   \end{equation}
   and all estimates in the theorem carry over to the case of~$\tilde{M}_n$.
 \end{remark*}

 A natural question arise whether one can guarantee a stronger mode of convergence in~\eqref{th_conv} and what assumptions are needed for this. We show below that provided further smoothness of the fields, indeed one can expect much more than just pointwise convergence. To alleviate unnecessary geometric complications we state the further result for one-dimensional domains only, namely~$E = [a,b]$.

 Denote the norm in the space of smooth functions~$C^{(m)}(E)$ by
 \begin{equation}
   \|f\|_m = \sum\limits_{l=0}^m \sup\limits_{x}\left|f^{(l)}(x)\right|,
 \end{equation}
 where~$f^{(l)}$ is the~$l$-th derivative of~$f$ and~$f^{(0)} \overset{\mathrm{def}}{=} f$. We have the following theorem.
 
\begin{theorem} 
  Let~$\{(\xi_n(x), x \in E)\}_{n \ge 0}$ be a sequence of stochastic processes such that their realizations are~$C^{m+1}(E)$-smooth almost surely and such that realizations of their~$(m+1)$-th derivatives admit a modulus of continuity~$\theta$ almost surely. Set
  \begin{equation}
    M_n \overset{\mathrm{def}}{=} \|\xi_n\|_{m+1} + \sup\limits_{x \ne y} \frac{\left|\xi_n^{(m+1)}(x)-\xi_n^{(m+1)}(y)\right|}{\theta(|x-y|)},
  \end{equation}
  and assume that~$(\xi_n(x),\{\mathcal{F}_n\}_{n\ge 0})$ is a martingale for every fixed~$x \in E$. If
  \begin{equation}
    \label{col_cond1}
    \sup\limits_n \mexp{M_n} < +\infty,
  \end{equation}
  then there exist a random field~$(\xi(x), x \in E)$ with almost sure~$C^{m+1}(E)$-smooth realizations and such that realizations of its~$(m+1)$-th derivative admit the modulus of continuity~$\theta$ almost surely; moreover
  \begin{equation}
    \|\xi_n-\xi\|_m \underset{n \to \infty}{\longrightarrow} 0
  \end{equation}
   almost surely.
 \end{theorem}
 \begin{proof}
 First, note that the dominated convergence theorem and~\eqref{col_cond1} imply that~$(\xi_n^{(l)}(x),\{\mathcal{F}_n\}_{n\ge 0})$ is a martingale for every~$x \in E$ and for~$l=0,1,\ldots, m+1$. Also, it is clear that~$(M_n,\{\mathcal{F}_n\}_{n\ge 0})$ is a submartingale which is bounded due to~\eqref{col_cond1}, thus
   \begin{equation}
     M_n \underset{n \to \infty}{\longrightarrow} M
   \end{equation}
   for some random variable~$M$ with expectation~$\mexp{M}<+\infty$.

   We proceed further by induction. Consider the base case~$m=0$ and note that the assumptions of Theorem~\ref{th1} are satisfied for~$(\xi_n^{(1)}(x), x \in E)$. Thus, there is a stochastic process~$(\xi^{(1)}(x) , x \in E)$ such that its realizations admit~$\theta$ almost surely, in particular they are almost sure continuous, and the convergence takes place
   \begin{equation}
     \label{eq_col1}
     (\xi_n^{(1)}(x) , x \in E) \underset{n \to \infty}{\longrightarrow} (\xi^{(1)}(x), x \in E)
   \end{equation}
   pointwise almost surely. Note that~$\xi^{(1)}$ does not mean the derivative of~$\xi$ because the latter field has not yet been defined; we use this notation for convenience. However, further on we indeed construct a process~$(\xi(x), x \in E)$ in such a way that its derivative is~$(\xi^{(1)}(x), x \in E)$.

   Now, since~$(\xi_n(a),\{\mathcal{F}_n\}_{n\ge 0})$ is a bounded martingale, by Doob's convergence theorem we can find a random variable~$\xi(a)$, $\mexp{\xi(a)}< +\infty$, such that
   \begin{equation}
     \label{eq_col1a}
     \xi_n(a) \underset{n \to \infty}{\longrightarrow} \xi(a)
   \end{equation}
   almost surely.

   Now, let us define~$(\xi(x), x \in E)$ by
   \begin{equation}
     \label{eq_col2}
     \xi(x) \overset{\mathrm{def}}{=} \xi(a) + \int \limits_{a}^{x} \xi^{(1)}(s)\, ds.
   \end{equation}

   Clearly, ~$(\xi(x), x \in E)$ is~$C^{1}(E)$-smooth almost surely and realizations of its first derivative admit~$\theta$ almost surely. It is left to prove the convergence.

   We have
   \begin{equation}
     \begin{aligned}
       \|\xi_n - \xi \|_0 &= \sup \limits_x \left|\xi_n(a)-\xi(a)+\int \limits_a^x (\xi_n^{(1)}(s) - \xi^{(1)}(s)) \, ds\right|\\
       &\le |\xi_n(a)-\xi(a)| +\int \limits_E \left|\xi_n^{(1)}(s) - \xi^{(1)}(s)\right| \, ds.
     \end{aligned}
   \end{equation}

   Being H\"{o}lder continuous, $\xi^{(1)}$ is bounded almost surely; also~$\xi_n^{(1)}$ is bounded uniformly in~$n$ almost surely since
 \begin{equation}
   \|\xi_n^{(1)}\|_0 \le M_n
 \end{equation}
 and~$M_n$ is an almost sure convergent sequence. Then the dominated convergence along with~\eqref{eq_col1} and~\eqref{eq_col1a} yield
 \begin{equation}
   \|\xi_n - \xi \|_0  \underset{n \to \infty}{\longrightarrow} 0
 \end{equation}
 almost surely. This completes the proof of the base case.

 Now let~$m\ge1$ and suppose that the claim holds for~$(m-1)$. Thus, we have for~$(\xi_n^{(1)}(x), x \in E)$ that
 \begin{equation}
   \label{col_conv_mm1}
   \|\xi_n^{(1)} - \xi^{(1)}\|_{m-1} \underset{n \to \infty}{\longrightarrow} 0
 \end{equation}
 almost surely for some process~$(\xi^{(1)}(x) , x \in E)$ with almost sure~$C^{(m)}(E)$-smooth realizations and such that realizations of the~$m$-th derivative admit~$\theta$ almost surely.

 Using the same definition for~$(\xi(x), x \in E)$ as in~\eqref{eq_col2} where~$\xi(a)$ is as in~\eqref{eq_col1a}, we see that realizations of this process are~$C^{(m+1)}(E)$-smooth almost surely and the~$(m+1)$-th derivative admits~$\theta$. To prove the convergence we notice that
 \begin{equation}
   \label{col_conv_m}
   \begin{aligned}
     \|\xi_n - \xi \|_m &= \|\xi_n - \xi \|_0 + \|\xi_n^{(1)} - \xi^{(1)} \|_{m-1} \\
     &\le |\xi_n(a)-\xi(a)| + \sup \limits_x \left|\int \limits_a^x (\xi_n^{(1)}(s) - \xi^{(1)}(s)) \, ds\right| + \|\xi_n^{(1)} - \xi^{(1)} \|_{m-1}\\
     &\le |\xi_n(a)-\xi(a)| + \int \limits_E |\xi_n^{(1)}(s) - \xi^{(1)}(s)| \, ds + \|\xi_n^{(1)} - \xi^{(1)} \|_{m-1}.
   \end{aligned}
 \end{equation}
 Then, the formulas~\eqref{eq_col1a} and \eqref{col_conv_mm1} yield
 \begin{equation}
   \|\xi_n - \xi \|_m \underset{n \to \infty}{\longrightarrow} 0,
 \end{equation}
 the integral term disappearing due to the almost surely uniform convergence of~$\xi_n^{(1)}$ to~$\xi^{(1)}$ by the inductive hypothesis. This concludes the proof.
\end{proof}
\bigbreak

\noindent \textbf{Acknowledgments.} The author is supported by the RFBR grants 14-01-90406, 14-01-00237 and the SFB 701 at Bielefeld University.

\end{document}